\documentclass[letterpaper, 11pt,  reqno]{amsart}

\usepackage{amsmath,amssymb,amscd,amsthm,amsxtra, esint}

\usepackage[implicit=true]{hyperref}

\allowdisplaybreaks[2]

\usepackage{color}

\definecolor{gr}{rgb}   {0.,   0.69,   0.23 }
\definecolor{bl}{rgb}   {0.,   0.5,   1. }
\definecolor{mg}{rgb}   {0.85,  0.,    0.85}
\definecolor{yl}{rgb}   {0.8,  0.7,   0.}
\definecolor{or}{rgb}  {0.7,0.2,0.2}

\newtheorem{theorem}{Theorem} [section]

\newtheorem{lemma}[theorem]{Lemma}
\newtheorem{proposition}[theorem]{Proposition}
\newtheorem{remark}[theorem]{Remark}

\newtheorem{definition}[theorem]{Definition}

\newtheorem*{ack}{Acknowledgments}

\DeclareMathOperator*{\intt}{\int}

\DeclareMathOperator*{\supp}{supp}

\newcommand{\noi}{\noindent}
\newcommand{\Z}{\mathbb{Z}}
\newcommand{\R}{\mathbb{R}}
\newcommand{\C}{\mathbb{C}}
\newcommand{\T}{\mathbb{T}}

\let\Re=\undefined\DeclareMathOperator*{\Re}{Re}
\let\Im=\undefined\DeclareMathOperator*{\Im}{Im}

\let\P= \undefined
\newcommand{\P}{\mathbf{P}}

\newcommand{\E}{\mathbb{E}}

\renewcommand{\L}{\mathcal{L}}

\newcommand{\N}{\mathcal{N}}
\newcommand{\NB}{\mathbb{N}}

\newcommand{\FL}{\mathcal{F}L} 

\renewcommand{\S}{\mathcal{S}}

\newcommand{\F}{\mathcal{F}}

\newcommand{\al}{\alpha}
\newcommand{\be}{\beta}
\newcommand{\dl}{\delta}

\newcommand{\Dl}{\Delta}
\newcommand{\eps}{\varepsilon}

\newcommand{\g}{\gamma}
\newcommand{\G}{\Gamma}
\newcommand{\ld}{\lambda}

\newcommand{\s}{\sigma}
\newcommand{\Si}{\Sigma}
\newcommand{\ft}{\widehat}

\newcommand{\wt}{\widetilde}
\newcommand{\cj}{\overline}
\newcommand{\dx}{\partial_x}
\newcommand{\dt}{\partial_t}

\newcommand{\embeds}{\hookrightarrow}

\newcommand{\ta}{\theta}

\renewcommand{\l}{\ell}
\renewcommand{\o}{\omega}
\renewcommand{\O}{\Omega}

\newcommand{\les}{\lesssim}

\newcommand{\jb}[1]
{\langle #1 \rangle}

\newcommand{\ind}{\mathbf 1}

\numberwithin{equation}{section}
\numberwithin{theorem}{section}

\usepackage{tikz} 
\usepackage{marginnote}
\usepackage{scalerel} 

\usetikzlibrary{shapes.misc}
\usetikzlibrary{shapes.symbols}
\usetikzlibrary{decorations}
\usetikzlibrary{decorations.markings}

\tikzset{
	dot/.style={circle,fill=black,draw=black,inner sep=0pt,minimum size=0.5mm},
	>=stealth,
	}

\tikzset{
	dot2/.style={circle,fill=black,draw=black,inner sep=0pt,minimum size=0.2mm},
	>=stealth,
	}

\tikzset{
	ddot/.style={circle,fill=white,draw=black,inner sep=0pt,minimum size=0.8mm},
	>=stealth,
	}

\tikzset{decision/.style={ 
        draw,
        diamond,
        aspect=1.5
    }}

\tikzset{dia2/.style
={diamond,fill=white,draw=black,inner sep=0pt,minimum size=1mm},
	>=stealth,
	}

\tikzset{dia/.style
={star,fill=black,draw=black,inner sep=0pt,minimum size=1mm},
	>=stealth,
	}

\tikzset{dia/.style
={diamond,fill=black,draw=black,inner sep=0pt,minimum size=1.3mm},
	>=stealth,
	}

\makeatletter
\def\DeclareSymbol#1#2#3{\expandafter\gdef\csname MH@symb@#1\endcsname{\tikz[baseline=#2,scale=0.15]{#3}}}
\def\<#1>{\csname MH@symb@#1\endcsname}
\makeatother

\begin{document}
\baselineskip = 14pt

\title[stochastic NLS with almost space-time white noise]
{Stochastic nonlinear Schr\"{o}dinger equation with almost space-time white noise}

\author[J.~Forlano, T.~Oh, and Y.~Wang]
{Justin Forlano, Tadahiro Oh, and Yuzhao Wang}

\address{
 Justin Forlano\\ Maxwell Institute for Mathematical Sciences
 and 
 Department of Mathematics\\
 Heriot-Watt University\\
 Edinburgh\\ 
 EH14 4AS\\
  United Kingdom}

\email{j.forlano@hw.ac.uk}

\address{
Tadahiro Oh\\
School of Mathematics\\
The University of Edinburgh\\
and The Maxwell Institute for the Mathematical Sciences\\
James Clerk Maxwell Building\\
The King's Buildings\\
 Peter Guthrie Tait Road\\
Edinburgh\\ 
EH9 3FD\\United Kingdom} 

\email{hiro.oh@ed.ac.uk}

\address{
Yuzhao Wang\\
School of Mathematics, 
University of Birmingham, 
Watson Building, 
Edgbaston, 
Birmingham\\
B15 2TT, 
United Kingdom}

\email{y.wang.14@bham.ac.uk}



\begin{abstract}

We study the  stochastic cubic nonlinear Schr\"odinger equation (SNLS)
with an additive noise on the one-dimensional torus.
In particular, we prove local well-posedness of the (renormalized) SNLS
when the noise is  almost space-time white noise.
We also discuss a notion of criticality in this stochastic context,
comparing the situation  
with the stochastic cubic heat equation (also known as the stochastic quantization equation).

\end{abstract}



%
\maketitle
%


%
%
%
%

\section{Introduction}
\label{SEC:intro}
\subsection{Stochastic nonlinear Schr\"odinger equation}

We consider the Cauchy problem of the following 
stochastic cubic nonlinear Schr\"odinger equation (SNLS) with an additive  noise
on  the one-dimensional torus $\T = \R/\Z$:
\begin{equation}
\begin{cases}\label{SNLS0}
i \dt u -   \dx^2 u + |u|^{2}u =  \phi \xi \\
u|_{t = 0} = u_0, 
\end{cases}
\qquad ( t, x) \in \R_+ \times \T, 
\end{equation}

\noi
where $\xi(t, x)$ denotes  a (Gaussian) space-time white noise\footnote{In view of the time reversibility of the deterministic nonlinear Schr\"odinger equation, 
one can also consider \eqref{SNLS0} on $ \R\times \T$
by extending 
the white noise $\xi$ onto  $\R\times \T$.
For simplicity, however, we only consider positive times in the following.
} on $\R_+ \times \T$ 
and $\phi$ is a bounded  linear operator on $L^2(\T)$.
 SNLS \eqref{SNLS0}
has a wide range of applications, ranging from nonlinear optics and plasma physics to solid state physics and quantum statistics
\cite{FKLT, Agrawal, YousefiK}.
In the context of nonlinear fiber optics \cite{NM, Agrawal}, 
the nonlinear Schr\"odinger equation\footnote{Namely, $\phi = 0$ in \eqref{SNLS0}.  See \eqref{NLS0} below.} 
 (NLS), when derived from the Maxwell equations, 
describes transmission of a signal along a fiber line,
where the roles of the variables $t$ and $x$ are switched
from the ``standard'' interpretation, 
namely, 
 in this particular application for nonlinear fiber optics, 
$t$ denotes the (rescaled) propagation distance and $x$ denotes the (rescaled and translated) time.\footnote{At least locally in time.}
See \cite{NM, Agrawal} for further details.
In the following, however, we stick to the standard convention,
namely, 
we always  refer to $x$ as the spatial variable in $\T$ and
$t$ as the temporal variable in $\R_+$
in the remaining part of this paper.

When $\phi$ is the identity operator, 
  the stochastic forcing in \eqref{SNLS0}  reduces to the space-time white noise $\xi$.
The Cauchy problem \eqref{SNLS0} in this case 
is of particular interest
in terms of applications 
\cite{FKLT, FKLMT, YousefiK}
as well as its analytical difficulty
since the problem is then {\it critical}.
See Subsection \ref{SUBSEC:critical} for a further discussion.

We say that $u$ is a solution to \eqref{SNLS0} if it satisfies the following mild formulation
(= Duhamel formulation):
\begin{align}
u(t) = S(t) u_0 + i \int_0^t S(t - t') |u|^2 u (t') dt' - i \int_0^t S(t - t') \phi \xi(dt'), 
\label{SNLS1}
\end{align}

\noi
where $S(t) = e^{- i t \dx^2}$ denotes the linear Schr\"odinger propagator.
The last term on the right-hand side of \eqref{SNLS1}
is the so-called stochastic convolution, representing the effect of the random forcing.
In the following, we set 
\begin{align}
\Psi(t) :=   \int_0^t S(t - t') \phi \xi(dt').
\label{stoconv1}
\end{align}

\noi
If  $\phi  \in HS(L^2; H^s)$, namely, it
is a Hilbert-Schmidt
operator from $L^2(\T)$ to $H^s(\T)$, 
then a standard argument \cite{DZ} shows that 
$\Psi \in C(\R_+; H^s(\T))$ almost surely.
When $\phi = \text{Id}$, namely when the noise is given by the space-time white noise $\xi$, we have 
$\Psi \in C(\R_+; H^s(\T))$ almost surely
if and only if $s < -\frac 12$.
This roughness (in space) of the stochastic convolution is the source
of difficulty in studying SNLS \eqref{SNLS0} with the space-time white noise.

Given $\phi \in HS(L^2; H^s)$ for some $s > \frac 12$, 
local well-posedness of \eqref{SNLS0} in $H^s(\T)$ 
easily follows from 
the algebra property of $H^s(\T)$
and the unitarity of the linear Schr\"odinger propagator $S(t)$ on $H^s(\T)$.
For lower regularities, however, 
one needs to employ 
the Fourier restriction norm method due to Bourgain \cite{BO1}.
In particular, it is shown in \cite{CM}
that \eqref{SNLS0} is locally\footnote{A standard application of Ito's lemma
combined with the conservation of the $L^2$-norm for the deterministic NLS \eqref{NLS0}
yields  an a priori bound
on the $L^2$-norm of a solution and thus global well-posedness of~\eqref{SNLS0}
in $L^2(\T)$.  See \cite{CM} for details.}
 well-posed in $L^2(\T)$,
provided that  $\phi  \in HS(L^2; L^2)$.
The argument in \cite{CM} is based on (a slight modification of)
the $L^2$-local theory by Bourgain \cite{BO1}
and controlling the stochastic convolution in the relevant $X^{s, b}$-norm
(see Lemma  \ref{LEM:stoconv} below).
See \cite{DDT04} for a related argument in the context of the stochastic KdV equation.
We also mention
the well-posedness results \cite{DD, OPW} of SNLS~\eqref{SNLS0} on 
 the Euclidean space $\R^d$,
 where the Strichartz estimates and the dispersive estimate play an important role.

Our main goal in this paper is to study \eqref{SNLS0} 
when $\phi$ is almost the identity operator.
Given $\al \in \R$, 
let $\phi$ be the Bessel potential of order $\al$ given by
\begin{align}
\phi = \jb{\dx}^{-\al} : = (1 - \dx^2)^{-\frac\al2}.
\label{phi1}
\end{align}

\noi
Namely, the operator $\phi$ in \eqref{phi1} is the Fourier multiplier operator
with the multiplier given by $\jb{n}^{-\al}$:
\[ \ft{\phi f}(n)  =  \jb{n}^{-\al}\ft f(n)\]

\noi
for  $n \in \Z$, 
where $\jb{\,\cdot\,} = (1 + |\cdot|^2)^\frac{1}{2}$.
Then, 
we prove that (a renormalized version of) SNLS \eqref{SNLS0} is locally well-posed, 
provided that $\al > 0$.
See Theorem \ref{THM:1} below
for a precise statement.
Note that
our main result (Theorem~\ref{THM:1})
handles the case of {\it almost} space-time white noise forcing
since the operator $\phi$ in \eqref{phi1}
reduces to the identity operator
  when $\al= 0$.

Let $\phi$ be as in \eqref{phi1}.
Then, it is easy to see that $\phi \in HS(L^2; H^s)$ if and only if 
\begin{align*}
s <  \al - \frac 12.
\end{align*}

\noi
In particular, when $\al > \frac 12$, 
the $L^2$ well-posedness theory in \cite{CM} is readily 
applicable and we conclude that \eqref{SNLS0} is globally well-posed in $L^2(\T)$
in this case.
When $\al \leq \frac{1}{2}$, however, 
 the stochastic convolution lies almost surely outside $L^2(\T)$
(for fixed $t \ne 0$),
which causes a serious issue in studying \eqref{SNLS0} with rough noises.

Before proceeding further, let us first discuss
the situation for 
 the  (deterministic) cubic nonlinear Schr\"odinger equation:
\begin{equation}
i \dt u -   \dx^2 u + |u|^{2}u =  0.
\label{NLS0}
\end{equation}

\noi
By introducing the Fourier restriction norm method, 
Bourgain \cite{BO1} proved that \eqref{NLS0} is locally well-posed
in $L^2(\T)$, which was immediately extended to global well-posedness thanks to 
the conservation of the $L^2$-norm.
On the other hand, 
it is known that \eqref{NLS0} is ill-posed in negative Sobolev spaces \cite{CCT, MOLI, GuoOh}.
In order to overcome this issue, 
the following renormalized NLS:
\begin{equation}
\textstyle i \dt u -   \dx^2 u +  \big(|u|^{2} -2 \int_\T |u|^2 dx\big)u  = 0
\label{NLS0a}
\end{equation}

\noi
has been proposed as an alternative model to \eqref{NLS0} outside $L^2(\T)$
\cite{C1, GH, CO, OS, GuoOh}.
We point out that \eqref{NLS0} and \eqref{NLS0a} are equivalent in $L^2(\T)$
in the sense that 
 the following invertible gauge transformation:
\begin{equation}
u(t) \longmapsto
\mathcal{G}(u)(t) : = e^{-  2 i t \int |u|^2dx} u(t),
\label{gauge}
\end{equation}

\noi
allows us to 
freely convert solutions to \eqref{NLS0} to those to \eqref{NLS0a},
provided that they belong to  $ C(\R; L^2(\T))$.
The renormalized NLS \eqref{NLS0a} first appeared 
in the work of Bourgain~\cite{BO96}
in studying  the invariant Gibbs measure for the defocusing cubic NLS on $\T^2$.
In \cite{BO96}, it was introduced as
 a model equivalent to the Hamiltonian dynamics
 corresponding to the Wick ordered Hamiltonian arising 
 in Euclidean quantum field theory. 
See \cite{OTh} for a further discussion on the Wick renormalization
in the context of NLS on $\T^2$.
For this reason, the equation~\eqref{NLS0a}
is often referred to as the Wick ordered NLS.
The gauge transformation \eqref{gauge}
 removes a certain singular component from the cubic nonlinearity in \eqref{NLS0}; 
 see \eqref{nonres} and \eqref{res} below.
 As a result, 
the Wick ordered NLS \eqref{NLS0a}  
 behaves better than  the cubic NLS \eqref{NLS0} outside $L^2(\T)$,
 while they are equivalent in $L^2(\T)$.
In particular, we proposed in \cite{OS} that 
the Wick ordered  NLS \eqref{NLS0a}
is the right model to study outside $L^2(\T)$.

In an analogous manner, 
we propose to study the following renormalized 
SNLS: 
  \begin{equation}
\begin{cases}\label{SNLS2}
i \dt u -   \dx^2 u + (|u|^2- 2\int_\T |u|^2 dx\big)u   =  \phi \xi \\
u|_{t = 0} = u_0, 
\end{cases}
\end{equation}

\noi
for rough noises: $\phi \notin HS(L^2; L^2)$.
For simplicity, we refer to \eqref{SNLS2} as
the Wick ordered SNLS in the following.
 Our main goal is to establish local well-posedness of \eqref{SNLS2}
for  $\phi$ given by \eqref{phi1} with $\al >0$
arbitrarily close to $0$.
For this purpose, we now go over the known results on 
 the Wick ordered NLS \eqref{NLS0a} outside $L^2(\T)$.
It was observed in \cite{CO} that 
the Wick ordered NLS \eqref{NLS0a} is mildly ill-posed
in negative Sobolev spaces in the sense
of the  failure of local uniform continuity of the solution map.
This in particular implies that 
we can not apply a contraction argument to construct solutions
to \eqref{NLS0a} in negative Sobolev spaces.
In~\cite{GuoOh}, 
the second author (with Z.\,Guo)
employed a more robust energy method (in the form of the short-time Fourier restriction norm method)
and
proved local\footnote{This local existence result in \cite{GuoOh} can be extended to 
global existence.  See  \cite{OW2, KVZ}.} existence of solutions (without uniqueness) to the Wick ordered NLS \eqref{NLS0a}
in $H^s(\T)$, $-\frac 18 < s < 0$.
This result leaves a substantial gap to the desired regularity $s \approx - \frac{1}{2}$, 
corresponding to $\al \approx  0$. 
More importantly, the question of uniqueness 
for the Wick ordered NLS \eqref{NLS0a} in negative Sobolev spaces still remains as  a very challenging open question in the field
of nonlinear dispersive PDEs.
Hence, the approach in \cite{GuoOh} does not seem to be suitable
for studying the Wick ordered SNLS \eqref{SNLS2}.

In \cite{CO}, 
the second author (with Colliander) studied the Wick ordered NLS \eqref{NLS0a} with 
random initial data of the form:\footnote{In the following, we may drop the harmless factor of $2\pi$ when it plays no important role.}
\begin{align}
u_0(x; \o) = u_0^\o(x) = \sum_{n \in \Z} \frac{g_n(\o)}{\jb{n}^\al}e^{2\pi i nx},
\label{gauss1}
\end{align}

\noi
where $\{g_n\}_{n \in \Z}$ is a sequence of independent standard complex-valued Gaussian random variables.
Denoting by $\eta$ the spatial white noise on $\T$, 
we have $u_0^\o = \jb{\dx}^{-\al} \eta$, 
namely, for $\al > 0$,  the random initial data $u_0^\o$ corresponds to a smoothed spatial  white noise.
By exploiting a gain of space-time integrability of the random linear solution $S(t) u_0^\o$, 
we proved that the Wick ordered NLS \eqref{NLS0a} is locally well-posed
almost surely with respect to the random initial data \eqref{gauss1}, provided that $\al > \frac 16$.
By slightly modifying the argument in \cite{CO}, 
we can show that the Wick ordered SNLS \eqref{SNLS2}
is locally well-posed with (deterministic) initial data  in $L^2(\T)$, 
provided that $\phi = \jb{\dx}^{-\al}$ for $\al > \frac 16$.
The main idea would be to apply  the so-called Da Prato-Debussche trick \cite{DPD}, 
namely, write $u = v + \Psi$ with $\Psi$ as in \eqref{stoconv1}
and study the fixed point problem for the residual term $v := u - \Psi$.
In order to lower the regularity of the noise below $\al = \frac 16$, however, 
we would need to employ higher order expansions
\cite{BOP3, OTzW}.
In particular, since our main goal is to  handle an almost white noise (i.e.~arbitrarily small $\al >0$), 
we would need to (at least) consider higher order expansions of arbitrarily high orders
(see \cite{OTzW}), 
which seems to be out of reach at this point.

In order to overcome this difficulty, we leave the realm of the $L^2$-based Sobolev spaces.
More precisely, 
we study the Wick ordered SNLS \eqref{SNLS2}
in the Fourier-Lebesgue spaces $\F L^{s, p}(\T)$
in the following.
Here,  the Fourier-Lebesgue space $\F L^{s, p}(\T)$ is defined by the norm:
\begin{align}
\|f\|_{\FL^{s,p}(\T)}:=\|\jb{n}^{s}\ft f(n)\|_{\l^{p}_{n}(\Z)}.
\label{FL1}
\end{align}

\noi
When $s = 0$, we simply set $\F L^{p}(\T) = \F L^{0, p}(\T)$.
Recall the following  embedding:
$\F L^{p_1}(\T) \subset L^2(\T) \subset \F L^{p_2}(\T)$ for $p_1 \leq 2 \leq p_2$.
%
Let us first go over the known results for the Wick ordered NLS  \eqref{NLS0a}
in the Fourier-Lebesgue spaces.
In~\cite{C1}, by a power series expansion, 
Christ constructed a solution to \eqref{NLS0a} (without uniqueness)
with initial data in 
 $\F L^p(\T)$, $1\leq p < \infty$.
In \cite{GH}, Gr\"unrock-Herr adapted the Fourier restriction norm method
to the Fourier-Lebesgue space setting (see Section \ref{SEC:space}) and proved local well-posedness
of~\eqref{NLS0a} in $\F L^p(\T)$ for $1\leq p < \infty$
by a standard contraction argument.
We also mention the work~\cite{ORW}
on the construction of solutions to
~\eqref{NLS0a}  in $\F L^p(\T)$, $1\leq p < \infty$, 
 based on a normal form method.
These  results allow us to handle almost spatial white noise,
i.e.~$u_0^\o$ in~\eqref{gauss1} with arbitrarily small $\al > 0$,
suggesting
that the Fourier-Lebesgue setting is an appropriate framework 
for studying the Wick ordered SNLS \eqref{SNLS2}.

In order to study the Wick ordered SNLS \eqref{SNLS2} in the Fourier-Lebesgue spaces, 
we first need to extend the notion of Hilbert-Schmidt operators
to the Banach space setting.
Given $s \in \R$ and $1\leq p < \infty$, we say that 
$\phi$ is a $\g$-radonifying operator
from $L^2(\T)$ to $\FL^{s, p}(\T)$
if the $\g(L^{2}(\T); \FL^{s,p}(\T))$-norm defined by 
\begin{equation}
\|\phi\|_{\g(L^{2}; \, \FL^{s,p})}
: =\bigg\| \Big( \sum_{k\in \Z} |\jb{n}^{s}\ft { \phi(e_k)}(n)|^{2} \Big)^{\frac{1}{2}}  \bigg\|_{\l^{p}_{n}(\Z)} 
\label{G1}
\end{equation}

\noi
is finite, 
where $e_k (x) = e^{2\pi i k x}$, $k \in \Z$.
We denote the collection of $\g$-radonifying operators
by $\g(L^{2}(\T); \FL^{s,p}(\T))$.
Note that when $p = 2$, the norm in \eqref{G1} reduces
to the standard Hilbert-Schmidt $HS(L^2;H^s)$-norm.
For readers' convenience, 
we present basic definitions and properties
of  $\g$-radonifying operators 
in  Appendix \ref{SEC:gamma}.

We now state our main result of this paper.

\begin{theorem}\label{THM:1}
Let $s > 0$ and $1< p < \infty$.
Then, given  $\phi\in \g(L^{2}(\T); \FL^{s,p}(\T))$, 
  the Wick ordered SNLS \eqref{SNLS2} is pathwise locally-well posed in $\FL^{s,p}(\T)$. 
More precisely, 
given $u_0\in \FL^{s,p}(\T)$,
there exists a stopping time $T=T(\|u_{0}\|_{\FL^{s,p}},\Psi)$, which is positive almost surely, 
and a unique solution  $u$ to \eqref{SNLS2}
in the class 
\[C([0,T]; \FL^{s,p}(\T))\cap X^{s,b}_{p}([0, T])\]

\noi
for some $b >0$ such that $(b-1) p < -1$.

\end{theorem}

Here, $X^{s, b}_p([0, T])$ denotes the local-in-time version of the $X^{s, b}$-space
adapted to $\FL^{s, p}(\T)$.
See Section \ref{SEC:space} for the precise definition.

Let $\phi = \jb{\dx}^{-\al}$ be as in \eqref{phi1}.
Then, a direct computation with \eqref{G1} shows that 
$\phi \in \g(L^{2}(\T); \FL^{s,p}(\T))$ if and only if
\begin{align}
s < \al - \frac{1}{p}.
\label{phi3}
\end{align}

\noi
Hence, given $\al > 0$, we can choose 
sufficiently small $s>0$ and large $p \gg 1$ such that \eqref{phi3} holds
and thus Theorem \ref{THM:1} is applicable.
This establishes local well-posedness of the Wick ordered SNLS \eqref{SNLS2}
with almost space-time white noise.

By writing \eqref{SNLS2} in the mild formulation, we have
\begin{align*}
u(t) = S(t) u_0 + i \int_0^t S(t - t') \N(u)(t') dt'  -i \Psi
\end{align*}

\noi
where $\N(u) = (|u|^2- 2\int_\T |u|^2 dx\big)u $
denotes the Wick ordered nonlinearity in \eqref{SNLS2}
and $\Psi$ denotes the stochastic convolution in \eqref{stoconv1}.
The proof of Theorem \ref{THM:1} is based on
a two-step argument:
(i) We construct a solution $u$ in  $X^{s, b}_p([0, T])$ by 
a standard contraction argument.
The local-in-time regularity of the stochastic convolution
(= the local-in-time regularity of the Brownian motion)
forces us to choose the temporal regularity  $b$
such that $(b - 1) p < -1$.
We point out that 
 Gr\"unrock-Herr \cite{GH} imposed a stronger regularity assumption: $(b - 1) p > -1$
 and therefore their trilinear estimate (Proposition 1.3 in \cite{GH}) is not directly applicable to our problem.
We make up this loss of temporal regularity (as compared to \cite{GH})
by  taking the spatial regularity $s$ to be slightly positive.
(ii)  For this particular choice of the temporal regularity, 
we have  $X^{s, b}_p([0, T]) \not\subset C([0, T]; \FL^{s, p}(\T))$.
Hence, we need to show a posteriori the continuity in time of the solution $u$ constructed in Step (i).

\begin{remark}\rm 
(i) Our main goal is to handle the case of almost space-time white noise, 
namely,  small $s>0$ and large $p \gg 1$ such that \eqref{phi3} holds.
As such, our proof of Theorem~\ref{THM:1} is tailored for this purpose.
For example, our argument does not treat the $p = 1$ case.
Note that, when $p = 1$,  local well-posedness of \eqref{SNLS2}
follows from the algebra property of $\FL^{1}(\T)$, the unitarity of $S(t)$ on $\FL^{1}(\T)$, 
and Lemma \ref{LEM:stoconv2} and there is no need to resort to the Fourier restriction norm method.

\smallskip

\noi
(ii) As we see in the next subsection, 
SNLS with an additive space-time white noise is critical.
While Theorem \ref{THM:1} establishes an almost critical local well-posedness result,
the problem with $\phi = \text{Id}$ seems to be out of reach at this point.

\smallskip

\noi
(iii)  
In the case of a spatially homogeneous noise, namely,
when $\phi$ is a convolution operator with a kernel $K$,  
then the  $\g(L^{2}(\T); \FL^{s,p}(\T))$-norm in \eqref{G1} reduces to the
$\FL^{s, p}$-norm of the kernel function $K$.
In \cite{Oh09b}, 
 the second author 
 used such a characterization of  $\g$-radonifying operators 
 and  proved local well-posedness
of the stochastic KdV equation with an additive space-time white noise.

\smallskip

\noi
(iv) 
Given $N \in \NB$, consider the following  SNLS with a truncated nonlinearity
(but with a full space-time white noise):
\begin{equation}
i \dt u -   \dx^2 u + \P_{N}\big(|\P_N u|^{2}\P_N u\big) =   \xi 
\label{SNLS4}
\end{equation}

\noi
where 
 $\P_N$ denotes the Dirichlet projection onto the frequencies $\{ |n|\leq N\}$.
Let $u$ be a solution  to \eqref{SNLS4}
with initial data given by the (spatial) white noise, 
i.e.~$u_0^\o$ in \eqref{gauss1} with $\al =0$.
Then, by exploiting invariance of the spatial white noise
under the deterministic NLS (with a truncated nonlinearity; see \cite{OQV}), 
we can show that the solution $u$ at time $t>0$
is given by the spatial white noise of variance $1+t$. 
The same result also holds for the Wick ordered SNLS with a truncated nonlinearity.
This would provide a basis for applying a modification of Bourgain's invariant measure argument
\cite{BO94, BO96} in constructing global-in-time
dynamics for  the Wick ordered SNLS \eqref{SNLS2} with $\phi = \text{Id}$.
See \cite{OQS} for details.
Unfortunately, we do not know how to construct local-in-time dynamics
for the Wick ordered SNLS \eqref{SNLS2} with $\phi = \text{Id}$.

\end{remark}

\begin{remark}\rm
In a recent paper  \cite{OW3}, 
the second and third authors 
exploited the completely integrable structure of the equation
and proved global well-posedness of~the renormalized (deterministic) NLS \eqref{NLS0a} 
in $\F L^{s, p}(\T)$ for $s\geq 0 $ and $ 1\leq p < \infty$. 
It would be of interest to investigate the global-in-time behavior of solutions
to the renormalized SNLS~\eqref{SNLS2} constructed in Theorem \ref{THM:1}.

\end{remark}

\subsection{On the criticality of SNLS with space-time white noise}
\label{SUBSEC:critical}

In this subsection, we discuss a  notion of criticality
for SNLS \eqref{SNLS1} with an additive space-time white noise forcing
(i.e.~$\phi = \text{Id}$); see also a discussion in \cite{BOP4}.
Before doing so, 
let us first go over 
Hairer's notion of local (sub)criticality \cite{Hairer}
by considering  the following stochastic cubic heat equation 
with an additive space-time white noise forcing
on the $d$-dimensional spatial domain:\footnote{We have $\T^d$
in mind but we intentionally remain vague about the underlying spatial domain for the purpose of scaling in this formal discussion.}
\begin{align}
 \dt u -  \Dl u + u^3 = \xi.
\label{SQE}
\end{align}

\noi
Here, $u$ is a real-valued function/distribution and  $\xi$ denotes a space-time white noise.
The equation \eqref{SQE} is also known as  the stochastic quantization equation (SQE).
The white noise scaling\footnote{Recalling that 
$\E[\xi(t_1, x_1) \xi(t_2, x_2)] = \dl(t_1 - t_2) \dl(x_1 - x_2)$, 
we see that the white noise
behaves like a square root of the Dirac delta function, which gives an intuition for 
the white noise scaling \eqref{sc1}. Moreover, in view of the linear part of the equation, 
we count one temporal dimension as two spatial dimensions.}
tells us that 
\begin{align}
 \xi_\ld(t, x) =  \ld^\frac{d+2}{2} \xi(\ld^2 t , \ld x)
 \label{sc1}
\end{align}

\noi
is also a space-time white noise.
By applying the following scaling to the unknown
\begin{align}
u_\ld(t, x) =  \ld^{\frac{d}{2}-1} u(\ld^2t, \ld x)
 \label{sc2}
\end{align}

\noi
we see that the scaled function $u_\ld$ satisfies the following equation:
\begin{align}
 \dt u_\ld -  \Dl u_\ld  + \ld^{4-d} u_\ld^3 = \xi_\ld.
\label{SQE2}
\end{align}

\noi
As $\ld \to 0$, namely, 
 studying  the behavior of a solution at smaller and smaller scales, 
we see that, at least at a formal level, the nonlinearity vanishes
and \eqref{SQE2} reduces to the following stochastic heat equation:
\begin{align*}
 \dt u -  \Dl u  = \xi, 
\end{align*}

\noi
provided that $ 4- d > 0$.
This formal discussion shows that SQE \eqref{SQE} in dimensions $d = 1, 2$, and $3$
is locally subcritical, while it is locally critical when $d = 4$.
See Section~8 in~\cite{Hairer} for a more rigorous definition of local subcriticality.
Indeed, when $d \leq 3$, 
SQE~\eqref{SQE} is known to be well-posed
(after an appropriate renormalization for $d = 2, 3$);
see  \cite{DPD2, Hairer, CC, MW}.
As we see below, 
while this notion of local criticality is suitable
for studying the heat equation, 
it  is not a suitable concept 
for studying the Schr\"odinger equation.

By repeating  a similar  scaling argument for 
SNLS with an additive space-time white noise:
\begin{align}
i \dt u -  \Dl u + |u|^2 u  = \xi
\label{XNLS1}
\end{align}

\noi
with \eqref{sc1} and \eqref{sc2}, 
we arrive at the scaled equation:
\begin{align}
i  \dt u_\ld -  \Dl u_\ld  + \ld^{4-d} |u_\ld|^2 u_\ld = \xi_\ld.
\label{XNLS2}
\end{align}

\noi
It is tempting to conclude that, by taking $\ld \to 0$, 
we may  neglect the effect of the nonlinearity when $d \leq 3$ as in the case of SQE.
This is,  however,  not quite correct since, 
in order to compare the sizes of the terms in \eqref{XNLS2}, 
we need to measure them in a norm compatible with the  Schr\"odinger equation.
For example, in terms of the $L^2$-based homogeneous Sobolev spaces, 
the $\dot H^{-\frac{d}{2}-}$-norm\footnote{We  use  $a-$ to denote  $a - \eps$ for arbitrarily small $\eps \ll 1$,
where a relevant norm diverges as $\eps \to 0$.} captures the (spatial) regularity of the white noise.
On the other hand, the scaling \eqref{sc2} preserves the $\dot H^1$-norm, 
while 
the $\dot H^{-\frac{d}{2}-}$-norm scales by a factor of $ \ld^{-\frac{d}{2}-1}$.\footnote{Strictly speaking, we have a factor of $ \ld^{-\frac{d}{2}-1+}$ here. We, however, use $\ld^{-\frac{d}{2}-1}$ for simplicity.
Recall also  the endpoint Besov regularity of the (spatial) white noise in 
$B^{-\frac{d}{2}}_{2, \infty}$; see \cite{Roy, BenyiOh}.}
By combining the factor $\ld^{4-d}$ in \eqref{XNLS2} with $(\ld^{-\frac{d}{2}-1})^2$, 
we essentially have $\ld^{2 - 2d}$ as the size of the nonlinearity (relative to the linear term $u_\ld$).
This shows that SNLS \eqref{XNLS1} with an additive space-time white noise 
is  {\it critical}  when $d = 1$.

Hairer's notion of local criticality
is useful for studying the heat equation since it is adapted
to the H\"older spaces $C^s = B^s_{\infty, \infty}$.
Recall that the stochastic convolution 
\[ \Psi_\text{heat} = \int_0^t e^{(t-t')\Dl} \xi(dt')\]

\noi
for the problem \eqref{SQE}
has the spatial regularity $- \frac d2 + 1-$ almost surely.
We expect a solution $u$ to \eqref{SQE} to have the  same regularity.
Noting that  the scaling \eqref{sc2} (essentially) preserves the 
relevant $\dot B^{-\frac{d}{2} + 1-}_{\infty, \infty}$-norm, 
we conclude from \eqref{SQE2} that SQE \eqref{SQE}
is critical when $d = 4$.
On the other hand, 
our discussion above on  a notion of criticality 
for SNLS was based on  the $L^2$-Sobolev spaces.
We may also carry out a similar analysis in terms of the homogeneous
Fourier-Lebesgue spaces $\dot \FL^{s, p}$
defined by the norm:
\begin{align*}
\|f\|_{\dot \FL^{s,p}(\R^d)}:=\big\||\zeta|^{s}\ft f(\zeta)\big\|_{L^{p}(\R^d)}.
\end{align*}

\noi
Note that the $\dot \FL^{s, p}$-norm with $s = -\frac{d}{p}-$
 captures the (spatial) regularity of the white noise.
One can easily check that,  
under the scaling~\eqref{sc2}, 
the $\dot \FL^{-\frac{d}{p}-, p}$-norm also scales by the factor of $ \ld^{-\frac{d}{2}-1}$
(just like the $H^{-\frac d2 -}$-norm), 
confirming the criticality of 
SNLS \eqref{XNLS1} with an additive space-time white noise 
 when $d = 1$.
The main point is that, 
in discussing a notion of criticality, 
we need to specify a function space
suitable for studying a given equation
and to incorporate the effect of the scaling on the size of the scaled function
(measured in the relevant norm). 
This point was not clearly mentioned in Hairer's presentation since 
the relevant H\"older norm is preserved under the scaling \eqref{sc2}.

\medskip

We can also argue in terms of the more standard scaling analysis.
It is well known that the cubic nonlinear Schr\"odinger equation on $\R^d$ remains invariant
under the following scaling symmetry:
\begin{align}
u^\ld(t, x) = \ld u( \ld^2 t, \ld x).
\label{sc3}
\end{align}

\noi
A direct computation shows
that the homogeneous $\dot W^{s_\text{crit}(p), p}$-norm is preserved under the scaling
\eqref{sc3}, where $s_\text{crit}(p)$ is given by 
\begin{align}
s_\text{crit}(p) = \frac{d}{p} - 1.
\label{crit1}
\end{align}

\noi
In studying NLS, we need to use the $L^2$-based Sobolev space since the linear Schr\"odinger 
propagator $e^{- it\Dl}$
is bounded on $L^2$ but is unbounded on $L^p$, $p \ne 2$.
This gives the (usual) scaling-critical Sobolev regularity:
\begin{align*}
s_\text{crit}(2) = \frac d2 -1.
\end{align*}

\noi
Now, consider SNLS \eqref{XNLS1} with an additive space-time white noise forcing.
Then, we see that 
\[ s_\text{crit}(2) = \frac d2 -1 > -\frac{d}{2}- = \text{ spatial regularity of the space-time white noise}\]

\noi
with almost an equality when $d = 1$.
This shows that the equation  \eqref{XNLS1} is critical when $d = 1$.
We point out that, when $d = 1$, 
 the (deterministic) Wick ordered NLS \eqref{NLS0a}
is known to be ill-posed at the critical regularity $ s_\text{crit}(2) = \frac d2 -1$ \cite{Kishimoto, O17, OW},
which shows the difficulty of the SNLS problem \eqref{SNLS2} with $\phi = \text{Id}$.

Similarly, noting  that the scaling-critical regularity (with respect to the scaling symmetry \eqref{sc3})
for the homogeneous Fourier-Lebesgue spaces $\dot \FL^{s, p}$
is given by 
\begin{align*}
\ft s_\text{crit}(p) = s_\text{crit}(p') = d - 1- \frac{d}{p} , 
\end{align*}

\noi
we have
\begin{align*}
\ft  s_\text{crit}(p) & = 
d - 1 - \frac{d}{p}   > -\frac{d}{p}- \\
&  = \text{ spatial regularity of the space-time white noise measured in } \dot \FL^{s, p}
\end{align*}

\noi
with almost an equality when $d = 1$.
Hence, 
we also conclude that 
the equation  \eqref{XNLS1} is critical when $d = 1$ in terms 
of the Fourier-Lebesgue spaces.

Lastly, let us consider SQE \eqref{SQE}.
In this case, the stochastic convolution $\Psi_\text{heat}$ gains one spatial derivative
and has spatial regularity 
\begin{align}
-\frac{d}{2} + 1 -.
\label{reg1}
\end{align}
On the other hand, we know that the linear heat propagator
$e^{t\Dl}$ is bounded in the $L^\infty$-based Sobolev spaces 
and the H\"older spaces $B^s_{\infty, \infty}$.
Namely, we can set $p = \infty$ in \eqref{crit1}, yielding
\begin{align}
s_\text{crit}(\infty) =  -1.
\label{crit4}
\end{align}

\noi
By comparing \eqref{reg1} and \eqref{crit4}, we obtain
\[ s_\text{crit}(\infty) =  -1 < -\frac{d}{2} + 1 -\]

\noi
for $d \leq 3$ and an almost equality holds when $d = 4$, showing
that the SQE problem \eqref{SQE} is critical when $d = 4$.

\medskip

This paper is organized as follows.
In Section \ref{SEC:space}, 
we recall the definition and basic properties
of the $X^{s, b}$-spaces adapted to the Fourier-Lebesgue spaces.
We then study the regularity properties of the stochastic convolution
in Section \ref{SEC:stoconv}.
In Section~\ref{SEC:nonlin}, we prove the crucial trilinear estimate,
which is then used to prove Theorem \ref{THM:1} in 
Section~\ref{SEC:THM}.
In Appendix \ref{SEC:gamma}, we go over basic definitions
and properties of $\g$-radonifying operators.

\section{Function spaces and their properties}\label{SEC:space}

Let $\S(\R\times \T)$ be the vector space 
of $C^{\infty}$-functions $u:\R^{2}\rightarrow \C$ 
such that
\[ u(t,x)=u(t,x+1) \quad 
\text{and} \quad\sup_{(t,x)\in \R^{2}}|t^{\al}\dt^{\be}\dx^{\g}u(t,x)|<\infty\]

\noi
for any $\al,\beta,\g\in \NB\cup\{0\}$.
In the seminal paper \cite{BO1}, Bourgain introduced the $X^{s, b}$-spaces
defined by the norm:
\begin{equation*}
\|u\|_{X^{s,b}(\R \times \T)}=\|\jb{n}^{s} \jb{\tau-n^{2}}^{b}
\ft u(\tau, n) \|_{\l_{n}^{2}L^{2}_{\tau}(\Z\times \R)}. 
\end{equation*}

\noi
We now recall the definition of the $X^{s, b}$-spaces  adapted to the Fourier-Lebesgue spaces;
see Gr\"unrock-Herr \cite{GH}.

\begin{definition} \rm
Let $s,b\in \R$, $1\leq p,q\leq \infty$. We define the space $X^{s,b}_{p,q}(\R \times \T)$ as the completion of 
$\S(\R\times \T)$ with respect to the norm 
\begin{equation*}
\|u\|_{X^{s,b}_{p,q}(\R \times \T)}=\|\jb{n}^{s} \jb{\tau-n^{2}}^{b}
\ft u(\tau, n) \|_{\l_{n}^{p}L^{q}_{\tau}(\Z\times \R)}. 
\end{equation*}
\end{definition}

For brevity, we simply  denote $X^{s,b}_{p, q}(\R \times \T)$ by $X^{s,b}_{p, q}$.
When $p = q$,  we set $X^{s,b}_{p}=X^{s,b}_{p, p}$. 
Recall the following characterization of the $X^{s, b}_{p, q}$-norm
in terms of the interaction representation $S(-t) u(t)$:
\begin{equation} 
\|u\|_{X^{s,b}_{p, q}}=\|S(-t)u(t)\|_{\FL^{s,p}_{x}\FL^{b,q}_{t}},
 \label{Xsb2}
\end{equation}

\noi
where the iterated norm is to be understood in the following sense:
 \begin{align}
 \|u\|_{\FL^{s,p}_{x}\FL^{b,q}_{t}}:=\|\jb{n}^{s} \jb{\tau}^{b}\ft u(\tau,n)\|_{\l^{p}_{n}L^{q}_{\tau}}
 =\| \|\jb{n}^{s}\ft u (t, n)\|_{\FL^{b,q}_{t}}\|_{\l^{p}_{n}}.
 \label{Xsb3}
 \end{align}
 
 \noi
 Note that these spaces are separable when $p,q<\infty$.
The identity \eqref{Xsb2} follows from
 $\ft{S(t) u}(\tau, n) = \ft u(\tau + n^2, n)$, the first equality in \eqref{Xsb3}, 
 and a change of variable: $\tau \mapsto \tau + n^2$.

For any $1\leq p < \infty$, $s\in \R$, we have 
\begin{equation}\label{conti}
X^{s,b}_{p, q} \embeds C(\R; \FL^{s,p}(\T)), \quad \text{if }  b>\frac{1}{q'} = 1 - \frac 1q.
\end{equation} 

\noi
This is a consequence of the dominated convergence theorem along with the following embedding relation:
 $\FL^{b, q}_t \embeds \FL^1_t \embeds C_{t}$,
where the second embedding 
is  the Riemann-Lebesgue lemma.

Given $T>0$, 
 we also define the local-in-time version $X^{s,b}_{p,q}([0,T])$ of 
the  $X^{s, b}_{p, q}$-space
as the collection of functions $u$ such that 
\begin{equation}
\|u\|_{X^{s,b}_{p, q}([0, T])}:=\inf \big\{\|v\|_{X^{s,b}_{p, q}} 
\, : \,  v|_{[0,T]}=u \big\}
\label{Xsb4}
\end{equation}

\noi
is finite.
For simplicity, we often denote
 $X^{s,b}_{p,q}([0,T])$ by $X^{s,b}_{p, q;T}$.

Lastly, we recall  the following linear estimates.

\begin{lemma}\label{LEM:lin} 
\textup{(i) (Homogeneous linear estimate).}
Given $1\leq p, q\leq \infty$  and $s,b\in \R$, we have 
\begin{equation*}
\|S(t)f\|_{X^{s,b}_{p, q;T}} \les \|f\|_{\FL^{s,p}}
\end{equation*}

\noi
for any $0 < T \leq 1$.

\medskip

\noi
\textup{(ii) (Nonhomogeneous linear estimate).}
Let $s\in \R$, $1\leq p\leq \infty$, $1<q<\infty$ and $-\frac{1}{q}<b'\leq 0\leq b \leq 1+b'.$ 
Then, we have 
\begin{equation}
\label{duhamel0}
\bigg\| \int_0^tS(t-t')F(t')dt'\bigg\|_{X^{s,b}_{p, q; T}}\les T^{1+b'-b}\|F\|_{X^{s,b'}_{p, q; T}}
\end{equation}

\noi
for any $0< T \leq 1$.

\end{lemma}

\begin{proof} 
The proof of (i) is straightforward from \eqref{Xsb2}.
More precisely, by letting  $\eta\in C^{\infty}_c(\R)$ be a smooth cutoff function 
such that 
$\eta(t)\equiv 1$ for $t\in [0,1]$ and 
$\supp \eta \subset [-2,2]$, 
it follows from \eqref{Xsb4} and then \eqref{Xsb2} that 
\begin{align*}
\|S(t)f\|_{X^{s,b}_{p, q;T}}
& \leq \|\eta (t) S(t)f\|_{X^{s,b}_{p, q}}
= \| \jb{n}^s \jb{\tau}^b \ft \eta(\tau) \ft f(n) \|_{\l^p_n L^q_\tau}\\
&  = \|\eta\|_{H^b_t} \|f\|_{\FL^{s,p}}.
\end{align*}

\noi
This proves (i).
The proof of the nonhomogeneous estimate \eqref{duhamel0}
is also standard.
From  (2.21) in \cite{Grock1}, we have
\begin{equation} \label{duhamel1}
\bigg \| \eta\left( \frac{t}{T}\right)\int_{0}^{t}f(t')dt'\bigg\|_{\FL^{b,q}_{t}}\les T^{1+b'-b}\|f\|_{\FL^{b',q}_{t}}.
\end{equation}

\noi
Given a function $F$ on  $[0,T]\times \T$, 
let $\wt{F}$ be an extension of $F$ onto $\R\times \T$. 
 Then,  it follows from  \eqref{Xsb2} and \eqref{duhamel1} that 
\begin{align*}
\bigg\| \int_{0}^{t}S(t-t')F(t')dt'\bigg\|_{X^{s,b}_{p, q; T}} & \leq \bigg\|\eta\left( \frac{t}{T}\right) \int_{0}^{t}S(t-t')\wt{F}(t')dt'\bigg\|_{X^{s,b}_{p,q}}  \\
& = \Bigg\|  \jb{n}^{s} \bigg\| \eta\left( \frac{t}{T}\right) \int_{0}^{t}
\F_x(S(-t')\wt{F})(t',  n)dt'   
\bigg\|_{\FL^{b,q}_{t}} \Bigg\|_{\ell^{p}_{n}}  \\
& \les T^{1+b'-b} \Big\|\jb{n}^{s}  \big\| \F_x(S(-t)\wt{F}(t))(n) \big\|_{\FL^{b',q}_{t}} \Big\|_{\ell^{p}_{n}}  \\
& =T^{1+b'-b}\|\wt{F}\|_{X^{s,b'}_{p,q}}.
\end{align*} 

\noi
Here, $\F_x$ denotes the spatial Fourier transform.
Taking the infimum over all  extensions $\wt{F}$ of $F$,
we obtain \eqref{duhamel0}. 
\end{proof}

\section{On the stochastic convolution} \label{SEC:stoconv}

In this section, we study the regularity properties of the stochastic convolution $\Psi$ defined in \eqref{stoconv1}.
For this purpose, let us first recall the definition of 
a cylindrical Wiener process $W$ on $L^2(\T)$;
a cylindrical Wiener process $W$ on $L^2(\T)$ is defined by the following random Fourier series:
\begin{align*}
 W(t) & =  \sum_{n\in\Z}
\be_n (t) e_n, 
\end{align*}

\noi
where $e_n(x) = e^{2\pi inx}$ and $\{ \be_n\}_{n \in \Z}$ is a family of  mutually independent complex-valued Brownian
motions.
In terms of the cylindrical Wiener process $W$, 
we can express the  stochastic convolution $\Psi$ in~\eqref{stoconv1}
as
\begin{align}
\Psi(t) =   \int_0^t S(t - t') \phi dW (t')
=   \sum_{n \in \Z}
\int_0^t S(t - t') \phi(e_n)  d\beta_n (t').
\label{S0}
\end{align}

It is easy to see that the space-time white noise almost surely belongs to 
 $\FL^{s, p}_{x}\FL^{b, q}_{t, \text{loc}}$
 if and only if $sp < -1$ and $bq < -1$.
The stochastic convolution, when measured in terms of its interaction representation $S(-t) \Psi(t)$, 
then gains one temporal regularity and thus has 
temporal regularity $b < 1 -\frac{1}{q}$.
This is precisely the regularity of the Brownian motion measured in the Fourier-Lebesgue spaces;
see \cite{BenyiOh}.
The next lemma tells us the regularity of the stochastic convolution with respect
to the $X^{s, b}_{p, q}$-spaces.

\begin{lemma}\label{LEM:stoconv} 
Let  $s \in \R$, $1\leq p<\infty$, $1<q<\infty$ such that  $ b<1-\frac{1}{q}$.
 Given $\phi \in \g(L^{2}(\T); \FL^{s,p}(\T))$, there exist $C, c>0$ such that
\begin{equation}
P\Big( \|\Psi\|_{X^{s,b}_{p,q;T}}>\ld\Big) 
\leq C\exp\bigg(-\frac{c\lambda^{2}}{
T^{3-2b- \frac{2}{q}}\|\phi\|_{\g(L^{2}; \FL^{s,p})}^{2}}\bigg)
\label{S0a}
\end{equation}

\noi
for any $\ld > 0$ and $0 < T \leq 1$.
In particular, $\Psi \in X^{s,b}_{p,q}([0, T])$ almost surely.
\end{lemma}

We first recall the following equivalence of moments for Gaussian random variables.
\begin{lemma}\label{LEM:gauss} 
Let $\{g_{n}\}_{n\in \Z}$ be a sequence of independent standard  complex-valued Gaussian random variables
on a probability space $(\O, \mathcal F, P)$. 
Then, for $ p \geq 2$, we have 
\[ \bigg\| \sum_{n\in \Z}a_{n}g_{n}\bigg\|_{L^{p}(\O)} \les \sqrt{p}\| a_{n}\|_{\l^{2}_{n}}\]

\noi
for any sequence $\{a_n\}_{n \in \Z} \in \l^2_n(\Z; \C)$.
\end{lemma}

Recall that a mean-zero complex-valued  Gaussian
random variable $g$ with variance $\s^2$ satisfies\footnote{This identity
follows from a simple computation involving the moment generating function
for $|g|^2$.
Suppose that $\s^2 = 1$.
On the one hand, we have
\[\E[e^{t|g|^2}]= \E[e^{t(\Re g)^2}]\E[e^{t(\Im g)^2}] =\frac{1}{1-t}
= \sum_{j = 0}^\infty t^j\]

\noi 
\noi
for $|t| < 1$.
On the other hand, the Taylor expansion gives
\[ \E[e^{t|g|^2}] = \sum_{k = 0}^\infty \frac{t^k}{k!} \E[|g|^{2k}].\]

\noi
By comparing the coefficients, we obtain
$\E\big[|g|^{2k}\big] =k!$ for any $k \in \NB$, 
when $\s^2 = 1$.
}
$\E\big[|g|^{2k}\big] =k!\cdot \s^{2k}$ for any $k \in \NB$.
In particular, from Stirling's formula, we obtain
$\| g\|_{L^p(\O)} \les \sqrt{p} \cdot \s$ for any $p \geq 2$.
Then, Lemma \ref{LEM:gauss} follows easily once we observe that 
$\sum_{n\in \Z}a_{n}g_{n}$ is a mean-zero complex-valued  Gaussian variable with variance
$\| a_{n}\|_{\l^{2}_{n}}^2$.
See also 
\cite[Theorem I.22]{Simon}.

\begin{proof}[Proof of Lemma \ref{LEM:stoconv}] 
The proof follows along the lines in \cite{DDT04, Oh09b}.
From \eqref{S0}, we have
\begin{equation*}
\Psi(t)= \sum_{n\in \Z}e_{n}\sum_{k\in \Z} \widehat{(\phi e_{k})}(n)\int_{0}^{t}e^{i(t-t')|n|^{2}}d\beta_k(t'). 
\end{equation*}

\noi
Then, we have
\begin{align}
S(-t) \ind_{[0,T]}(t)\Psi(t)& = 
\sum_{n\in \Z}e_{n} \sum_{k\in \Z} 
\ft {(\phi e_{k})}(n)\ind_{[0,T]}(t)\int_{0}^{t}\ind_{[0,T]}(t')e^{-it'|n|^{2}}d\beta_k(t') \notag\\
 & = \sum_{n \in \Z}\psi_{n}(t)e_{n}, 
 \label{S2}
 \end{align}

 \noi
 where 
 \begin{align*}
 \psi_{n}(t):=\sum_{k\in \Z}\ft{(\phi e_k)}(n)\ind_{[0,T]}(t)\int_{0}^{t}\ind_{[0,T]}(t')e^{-it'|n|^{2}}d\beta_k(t').
 \end{align*}

\noi
By the stochastic Fubini theorem, we have
 \begin{align}
 \mathcal{F}_t(\psi_{n})(\tau)
 = \sum_{k\in \Z}\int_{0}^{T} e^{-it'|n|^{2}}\ft{(\phi e_{k})}(n)
 \bigg(\int_{t'}^T \ind_{[0,T]}(t)e^{-it\tau}dt\bigg) d\beta_k(t').
  \label{S2a}
  \end{align}

For fixed $T, \tau$ and $n$, the summands in \eqref{S2a} are mean-zero independent complex-valued 
Gaussian random variables. 
Then,  for fixed $T, \tau$ and $n$, 
it follows from Lemma \ref{LEM:gauss} and the properties of the Wiener integral that, 
for $\s \geq 2$,  
\begin{align}
\|\F_t (\psi_{n})(\tau)  \|_{L^{\s}(\O)} 
&\sim \s^{\frac{1}{2}} \left( \sum_{k\in \Z} \E \Bigg[ \bigg| \int_{0}^{T} e^{-it'|n|^{2}}
\ft{(\phi e_{k})}(n)\bigg(\int_{t'}^T \ind_{[0,T]}(t)e^{-it\tau}dt\bigg) d\beta_{k}(t') \bigg|^{2}  
\Bigg] \right)^{\frac{1}{2}} \notag \\
 & \sim \s^{\frac{1}{2}} \Bigg( \sum_{k\in \Z} 
 \int_{0}^{T} |\ft{(\phi e_{k})}(n)|^{2}  
 \bigg| \int_{t'}^T \ind_{[0,T]}(t)e^{-it\tau}dt \bigg|^{2}dt'  \Bigg)^{\frac{1}{2}} \notag \\
 & \les \s^{\frac{1}{2}} T^{\frac{1}{2}}\min (T,|\tau|^{-1})
 \bigg( \sum_{k\in \Z}   |\ft{(\phi e_{k})}(n)|^{2}    \bigg)^{\frac{1}{2}}.
 \label{S3}
 \end{align}

 \noi
 Hence, from 
 \eqref{Xsb2}, \eqref{S2}, and \eqref{S3} with Minkowski's integral inequality
 and \eqref{G1}, 
 we obtain 
\begin{align}
\big\| \|\ind_{[0,T]}(t)\Psi(t)\|_{X^{s,b}_{p,q}}\big\|_{L^{\s}(\O)} 
& = \big\| \jb{n}^{s}\jb{\tau}^{b}\|\F_t(\psi_{n})(\tau) \|_{L^{\s}(\O)} \big\|_{\ell^{p}_{n}L^{q}_{\tau}} \notag \\
& \les   \s^{\frac{1}{2}} T^{\frac{1}{2}} 
\| \jb{\tau}^{b} \min (T,|\tau|^{-1})\|_{L^{q}_{\tau}} 
\bigg\| \Big( \sum_{k\in \Z}   |\jb{n}^s\ft{(\phi e_{k})}(n)|^{2}    \Big)^{\frac{1}{2}}\bigg\|_{\l^p_n}
\notag \\
& \les \s^{\frac{1}{2}} 
T^{\frac{3}{2}-b- \frac 1q}
\|\phi\|_{\g(L^{2}; \FL^{s,p})}
\label{S4}
\end{align} 

\noi
for any $\s \geq \max(p, q)$, provided that $b < 1 - \frac 1q$.
Therefore, \eqref{S0a} follows from \eqref{S4} and 
 Chebyshev's inequality.
\end{proof}

Next, we show the continuity in time of the stochastic convolution. 
This is necessary as Lemma~\ref{LEM:stoconv}, while useful for the fixed point argument, does not  imply any continuity in time since $b<1-\frac{1}{q}$; see \eqref{conti}.

\begin{lemma}\label{LEM:stoconv2} 
Let  $s \in \R$ and  $1\leq p < \infty$.
Then, given  $\phi \in \g(L^{2}(\T); \FL^{s,p}(\T))$, 
we have $\Psi \in C(\R_+; \FL^{s, p}(\T))$ almost surely.

\end{lemma}

\begin{proof}
Let $\wt{\Psi}(t):=S(-t)\Psi(t).$ 
Note that $\|\wt{\Psi}(t)\|_{\FL^{s,p}}=\|\Psi(t)\|_{\FL^{s,p}}$ for all $t \geq 0$. 
Then for $t_{2}>t_{1} \geq0$ and $\s\geq p$, Lemma \ref{LEM:gauss} and \eqref{G1} imply
 \begin{align*}
\E \Big[ \|\wt{\Psi}(t_{2})-\wt{\Psi}(t_{1})\|_{\FL^{s,p}}^{\s} \Big] 
& \les \Bigg(\sum_{n\in \Z}\jb{n}^{sp} \bigg \|\sum_{k\in \Z} \ft{(\phi e_{k})}(n)
\int_{t_{1}}^{t_{2}}e^{-it'|n|^{2}}d\beta_{k}(t')  \bigg\|_{L^{\s}(\O)}^{p}\Bigg)^{\frac{\s}{p}} \\
& \sim \s^{\frac{\s}{2}} \Bigg(\sum_{n\in \Z}\jb{n}^{sp} 
\bigg(\sum_{k\in \Z} |\ft{(\phi e_{k})}(n)|^{2} \int_{t_{1}}^{t_{2}} 1\,  dt'  \bigg)^{\frac{p}{2}}\Bigg)^{\frac{\s}{p}} \\
& \les \s^{\frac{\s}{2}}\|\phi\|_{\g(L^{2}; \FL^{s,p})}^{\s}|t_{2}-t_{1}|^{\frac{\s}{2}}.
\end{align*}

\noi
By applying Kolmogorov's continuity criterion 
 (\cite[Exercise 8.2]{Bass}), 
we see that   $\wt{\Psi}\in C(\R; \FL^{s,p}(\T))$ almost surely. 
 By continuity of the semigroup $S(t)$ on $\FL^{s,p}(\T)$ for $1 \leq p < \infty$, 
 we conclude that $\Psi(t)\in C(\R_+; \FL^{s,p}(\T))$ almost surely. 
\end{proof}

\section{Nonlinear estimates} \label{SEC:nonlin}

In this section, we establish  crucial nonlinear estimates
for proving Theorem \ref{THM:1}.
First, write  the Wick ordered nonlinearity in \eqref{SNLS2}
as follows:
\begin{align*}
\N(u) & = \bigg(|u|^2- 2 \int_\T |u|^2 dx\bigg)u \\
& = 
\sum_{n\in \Z} e^{inx}\sum_{\substack{n=n_1-n_2+n_3 \\ n \ne n_{1},n_{3}}} 
\ft u(n_{1})\cj{\ft u(n_{2})}\ft u(n_{3}) 
 -\sum_{n\in \Z} e^{inx}|\ft u(n)|^2\ft u(n)\\
 & =: \N_1(u) + \N_2(u).
\end{align*}

\noi
Here, $\N_1(u)$ and $\N_2(u)$ correspond to the non-resonant and resonant parts
of the nonlinearity, respectively.
By viewing $\N_1(u)$ and $\N_2(u)$ as trilinear operators, we write
\begin{align} 
\N_1(u_{1},u_{2},u_{3})& := 
\sum_{n\in \Z} e^{inx}\sum_{\substack{n=n_{1}-n_{2}+n_{3} \\ n \ne n_{1},n_{3}}} \ft u_{1}(n_{1})\overline{\ft u_{2}(n_{2})}\ft u_{3}(n_{3}), 
\label{nonres}\\
\N_2(u_{1},u_{2},u_{3})
& := -\sum_{n\in \Z} e^{inx}\ft u_{1}(n)\overline{\ft u_{2}(n)}\ft u_{3}(n).
\label{res}
\end{align}

\noi
Then, we have the following nonlinear estimates.

\begin{proposition} \label{PROP:nonlin} 
Let $s> 0$ and $1 < p < \infty$.
Then, there exist  $b, b' \in \R$ with  $-\frac{1}{p} < b' <0<b<1-\frac{1}{p}$
such that 
\begin{equation}
\|\N_1 (u_{1},u_{2},u_{3})\|_{X^{s,b'}_{p;T}}+ \|\N_2 (u_{1},u_{2},u_{3})\|_{X^{s,b'}_{p;T}}
\les \prod_{j=1}^{3}\|u_{j}\|_{X^{s,b}_{p;T}}.
\label{nonlin}
\end{equation}

\end{proposition}

Gr\"unrock-Herr \cite{GH} proved an analogous trilinear estimate 
but with  higher temporal regularity $b > 1 - \frac{1}{p}$ (and $s = 0$).
See Proposition 1.3 in \cite{GH}.
In our case, the temporal regularity of the stochastic convolution (Lemma \ref{LEM:stoconv})
forces us to work with lower temporal regularity: $b < 1 - \frac 1p$.
We compensate this loss of temporal regularity by assuming
a slightly higher spatial regularity $s > 0$.

Before proceeding to the proof of Proposition \ref{PROP:nonlin}, recall the following
elementary calculus lemma.  See, for example, Lemma 4.2 in \cite{GTV}.

\begin{lemma}\label{LEM:sum} 
Let $\be\geq \g \geq 0$ and $\be+\g >1$. 
Then, we have
\begin{align*}
 \int_{\R}   \frac{1}{\jb{x-a_1}^{\beta}\jb{x-a_2}^{\g}} 
&  \les \frac{1}{\jb{a_1-a_2}^{\al}}, \\
\sum_{n\in \Z}\frac{1}{\jb{n-k_{1}}^{\beta}\jb{n-k_{2}}^{\g}}
& \les \frac{1}{\jb{k_{1}-k_{2}}^{\al}}, 
\end{align*}

\noi
where $\al$ is given by 
\[ \al=
\begin{cases} 
\g, &  \beta>1, \\ 
\g-\eps, &  \beta=1, \\ 
\be+\g-1, &  \beta<1
\end{cases}
\]

\noi
for any $\eps > 0$.

\end{lemma}

Recall also the following arithmetic fact \cite{HW}.
Given $n \in \NB$, the number $d(n)$ of the divisors of $n$
satisfies
\begin{align}
d(n) \leq C_\dl \, n^\dl
\label{divisor}
\end{align}

\noi
for any $\dl > 0$.

\begin{proof}
Let $\wt{u}_{j}$ be an extension of $u_{j}$, $j = 1, 2, 3$.
Then, it suffices to prove 
\begin{align}
\|\N_k(\wt{u}_1,\wt{u}_2,\wt{u}_3)\|_{X^{s,b'}_{p}}\les  \prod_{j=1}^{3}\|\wt{u}_{j}\|_{X^{s,b}_{p}}
\label{N1}
\end{align}

\noi
for $k = 1, 2$,
since the desired estimate \eqref{nonlin} 
then follows from  taking an infimum over all  extensions. 
For simplicity, we denote $\wt u_j$ by $u_j$ in the following.

We first estimate the non-resonant part $\N_1(u_1, u_2, u_3)$.
Let $f_{j}(\tau, n):=\jb{n}^{s}\jb{\tau-n^{2}}^{b}|\ft u_{j}(\tau, n)|$ for $j=1,2,3$.
By noting that 
 $\|f_{j}\|_{\l_{n}^{p}L^{p}_{\tau}}  =\|u_{j}\|_{X^{s,b}_{p}}$, 
 we see that \eqref{N1} follows once we prove 
\begin{equation} 
\bigg\| \frac{\jb{n}^{s}}{\jb{\s_{0}}^{a}} 
\sum_{\substack{n=n_{1}-n_{2}+n_{3} \\ n \ne n_{1},n_{3}}} 
\intt_{\tau=\tau_{1}-\tau_{2}+\tau_{3}} 
\prod_{j=1}^{3} \frac{f_{j}(\tau_{j}, n_j)}{\jb{n_{j}}^{s}\jb{\s_{j}}^{b}} 
d\tau_{1}d\tau_{2} \bigg\|_{\l^{p}_{n}L^{p}_{\tau}}
\les \prod_{j=1}^{3} \| f_{j}\|_{\ell^{p}_{n}L^{p}_{\tau}},
\label{N2}
\end{equation}

\noi
where $a: =-b'>0$,  $\s_{0}:=\tau-n^{2}$,  and $\s_{j}:=\tau_{j}-n_{j}^{2}$, $j=1,2,3$.
By H\"older's inequality, we have
\begin{align*}
\text{LHS of }\eqref{N2} \leq \Big(\sup_{n, \tau} M_{n, \tau} \Big)^\frac{1}{p'} 
\prod_{j=1}^{3} \| f_{j}\|_{\ell^{p}_{n}L^{p}_{\tau}}, 
\end{align*}

\noi
where $M_{n, \tau}$ is defined by 
\[M_{n, \tau} := \frac{\jb{n}^{sp'}}{\jb{\s_{0}}^{ap'}} 
\sum_{\substack{n=n_{1}-n_{2}+n_{3} \\ n\ne n_{1},n_{3}}} \intt_{\tau=\tau_{1}-\tau_{2}+\tau_{3}} \prod_{j=1}^{3} \frac{1}{\jb{n_{j}}^{sp'} \jb{\s_{j}}^{bp'}} \, d\tau_{1}d\tau_{2}.
\]

\noi
Hence, it suffices to show that 
\begin{align*}
\sup_{n, \tau} M_{n, \tau}  < \infty.
\end{align*}

By the triangle inequality, we have 
 \begin{align} \label{N4}
\jb{n-n_{1}}\jb{n-n_{3}} 
&\sim \jb{n^{2}-n_{1}^{2}+n_{2}^{2}-n_{3}^{2}}
 \leq \sum_{j=0}^{3} \jb{\s_{j}} 
 \leq \prod_{j=0}^{3}\jb{\s_{j}}
\end{align} 

\noi
under $n=n_{1}-n_{2}+n_{3}$ and $\tau=\tau_{1}-\tau_{2}+\tau_{3}$.
By symmetry, assume  $|n_{1}|\geq |n_{3}|$. 
By the triangle inequality once again, we have
\begin{align}
 \frac{\jb{n}}{\jb{n_{1}}\jb{n_{2}}\jb{n_{3}}} \les \frac{1}{\jb{n_{1}}\jb{n_{3}}} \qquad \text{or} \qquad \les \frac{1}{\jb{n_{2}}\jb{n_{3}}}.
 \label{N5}
\end{align}

\noi
We  assume that the latter holds.
 The situation for the former is essentially identical to what we detail for the latter.
From \eqref{N4} and \eqref{N5}, we have
\begin{align*}
 M_{n, \tau} 
 & \les \sum_{\substack{n=n_{1}-n_{2}+n_{3} \\ n \ne n_{1},n_{3}}} 
 \frac{1}{\jb{n_{2}}^{sp'}\jb{n_{3}}^{sp'}} \frac{1}{\jb{n-n_{1}}^{ap'}\jb{n-n_{3}}^{ap'}} \notag\\
& \hphantom{X} 
\times \intt_{\tau=\tau_{1}-\tau_{2}+\tau_{3}} \prod_{j=1}^{3} \frac{1}{\jb{\s_{j}}^{(b-a)p'}}  d\tau_{1}d\tau_{2}.
\end{align*}

By choosing $b < \frac{1}{p'}$ sufficiently close to 
$\frac 1{p'}$ 
and $-\frac{1}{p} < b' < 0$ sufficiently close to 0, we have 
 \begin{align}
\frac 23 < (b-a)p' = (b+b') p'  < 1.
 \label{N6}
 \end{align}

\noi
Then,  applying  Lemma \ref{LEM:sum} twice, we have
\begin{align*} 
M_{n, \tau}
 &\les \sum_{\substack{n=n_{1}-n_{2}+n_{3} \\ n \ne n_{1},n_{3}}} 
 \frac{1}{\jb{n_{2}}^{sp'}\jb{n_{3}}^{sp'}} \frac{1}{\jb{n-n_{1}}^{ap'}\jb{n-n_{3}}^{ap'}}\\
 & \hphantom{XXX}
\times 
 \frac{1}{\jb{\s_{0}+2(n-n_{1})(n-n_{3})}^{3(b-a)p'-2}}.
\end{align*}

\noi
Fix $q \in [1, \infty]$ to be chosen later.
Then, by H\"{o}lder's inequality, we have
\begin{align*} 
M_{n, \tau}
& \les \bigg( \sum_{n_{2},n_{3}}\frac{1}{\jb{n_{2}}^{sp'q'}\jb{n_{3}}^{sp'q'}}\bigg)^{\frac{1}{q'}} 
\bigg( \sum_{\substack{n_{2},n_{3}\\ n_3 \ne n_2, n}}\frac{1}{\jb{n_{2}-n_{3}}^{ap'q}\jb{n-n_{3}}^{ap'q}}
\notag \\
 & \hphantom{XXXXXXXXXXX} 
\times 
\frac{1}{\jb{\s_{0}-2(n_{2}-n_{3})(n-n_{3})}^{(3bp'-3ap'-2)q}} \bigg)^{\frac{1}{q}}.
\end{align*}

\noi
The first factor on the right-hand side is finite, provided that 
\begin{align}
sp'q'>1.
\label{N8}
\end{align} 

\noi
As for the second factor, we set $k_{1}=n_{2}-n_{3}$, $k_{2}=n-n_{3}$ and $h=k_{1}k_{2}$.
By the divisor estimate \eqref{divisor}, 
we have
\begin{align}
\sum_{h\neq 0} \frac{1}{\jb{\s_{0}-2h}^{(3bp'-3ap'-2)q}}\frac{1}{\jb{h}^{ap'q}}\sum_{\substack{k_{1},k_{2}\neq 0\\ h=k_{1}k_{2} }} 1 \les \sum_{h\neq 0} \frac{1}{\jb{\s_{0}-2h}^{(3bp'-3ap'-2)q}}\frac{1}{\jb{h}^{(ap'-\ta)q}},
\label{N9}
\end{align}

\noi
for any  $\ta>0$. 
By applying Lemma \ref{LEM:sum} with \eqref{N6}, we see that the sum \eqref{N9} is finite
with a bound independent of $\s_0 \in \R$, provided that 
\begin{equation}
(3bp'-2ap'-2-\ta)q>1.
\label{N10}
\end{equation}

\noi
Putting \eqref{N8} and \eqref{N10} together, we obtain the restriction:
\begin{equation*}
3bp'-2ap'-2-\ta >1 - sp', 
\end{equation*}

\noi
which holds true for any $s > 0$ and $1<  p < \infty$ by choosing 
(i)  $b < \frac{1}{p'}$ sufficiently close to 
$\frac 1{p'}$, (ii)  $-\frac{1}{p} < b' = -a< 0$ sufficiently close to 0, 
and (iii) $\ta > 0$ sufficiently small.

\medskip

Next, we consider the resonant part $\N_2(u_1, u_2, u_3)$.
In this case, we prove the desired estimate with $s = 0$ and $b' = 0$.
A general case $s>0$ then follows from the triangle inequality: $\jb{n}^s \les \jb{n_1}^s\jb{n_2}^s\jb{n_3}^s$
for $ s\geq 0$. By Young's inequality, 
H\"older's inequality \big($\frac{2p+1}{3p} = \frac{1}{p} + \frac{2p-2}{3p}$\big), 
and $\l^p_n \subset \l^{3p}_n$, we have
\begin{align*}
\|\N_2(u_1,u_2,u_3)\|_{X^{0,0}_{p}}
& \leq \bigg\| \prod_{j=1}^{3} \| \ft u_j(t, n) \|_{\FL_t^{\frac{3p}{2p+1}}}\bigg\|_{\ell^{p}_{n}}
\les \bigg\| \prod_{j=1}^{3} \| \ft u_j(t, n) \|_{\FL_t^{b, p}}\bigg\|_{\ell^{p}_{n}}\\%
& \les  \prod_{j=1}^{3}\|u_{j}\|_{X^{0,b}_{p}},
\end{align*}

\noi
provided that $b > \frac{2p-2}{3p}$.
The last condition obviously holds  by taking $b < \frac 1{p'}$ sufficiently close to $\frac 1{p'}$
as long as $p > 1$.
\end{proof}

\section{Proof of Theorem \ref{THM:1}}
\label{SEC:THM}

In this section, we present the proof of Theorem \ref{THM:1}.
Given $s>0$ and  $1< p<\infty$,
fix $\phi\in \g(L^{2}(\T); \FL^{s,p}(\T))$. 
Given  $u_{0}\in \FL^{s,p}(\T)$,
define the operator $\G_{u_0}$ by 
\begin{align*}
\G_{u_0} u(t) := S(t) u_0 + i \int_0^t S(t - t') \N(u)(t') dt'  -i \Psi.
\end{align*}

Let  $b=\frac{1}{p'}-\dl$ for some $\dl>0$ sufficiently small
be given by  Proposition \ref{PROP:nonlin}. 
Then, by Lemma \ref{LEM:lin} with $b'=-\frac{1}{p}+\dl$ and  Proposition \ref{PROP:nonlin}, we have
\begin{align}
\|\G_{u_{0}}(u)\|_{X^{s,b}_{p;T}}  
& \leq C_1\|u_{0}\|_{\FL^{s,p}}+C_2T^{2\delta}\|u\|_{X^{s,b}_{p;T}}^{3} + \|\Psi\|_{X^{s,b}_{p;T}}
\label{X2}
\end{align}

\noi
for $0 < T \leq 1$.
Similarly, we have 
\begin{equation}
\|\G_{u_{0}}(u)-\G_{u_{0}}(v)\|_{X^{s,b}_{p;T}} 
\leq C_3T^{2\dl}
\Big(\|u\|_{X^{s,b}_{p;T}}^{2}+\|v\|_{X^{s,b}_{p;T}}^{2} \Big)\|u-v\|_{X^{s,b}_{p;T}}.
\label{X3}
\end{equation}

\noi
It follows from 
 Lemma~\ref{LEM:stoconv} that there exists a set $\Si \subset \O$
 with $P(\Si) = 1$ such that $ \Psi = \Psi^\o \in X^{s,b}_{p}([0, 1])$ for each $\o \in \Si$.
Now, choose\footnote{Here, we used the $X^{s,b}_{p}$-norm of $\Psi^\o$ on the interval $[0, 1]$
so that it does not depend on $T$, since it is used to determine $R^\o$, 
which in turn determines $T^\o$.}
$R^\o :=2C_1\|u_{0}\|_{\FL^{s,p}}+2\|\Psi^\o\|_{X^{s,b}_{p}([0, 1])}$
and positive $T = T^\o= T(R^\o)\ll 1$ such that 
\[C_2T^{2\dl}R^2\leq \frac{1}{2} \qquad \text{and} \qquad C_3T^{2\dl}R^2\leq \frac{1}{4},\]

\noi
Then, it follows from \eqref{X2} and \eqref{X3}
that  $\G_{u_{0}}$ is a  contraction on the closed ball $B_{R}\subset X^{s,b}_{p}([0, T])$ 
of radius $R$
and thus has  a unique fixed point $u=\G_{u_{0}}(u)\in X^{s,b}_{p}([0, T])$ for any $\o\in \Si$.

Our choice of $b = \frac{1}{p'} - \dl$ does not allow us to conclude
the continuity of $u$ in time in a direct manner.
From Lemma \ref{LEM:lin} (i) with \eqref{conti}, we see that the linear solution $S(t)u_0$
belongs to $ C([0,T];\FL^{s,p}(\T))$. 
From Lemma \ref{LEM:stoconv2}, we also have 
$\Psi\in C([0,T];\FL^{s,p}(\T))$ almost surely.
Finally, by applying Lemma \ref{LEM:lin} (ii) with  $b'=-\frac{1}{p}+2\delta>-\frac{1}{p}$
and Proposition \ref{PROP:nonlin}, we obtain
\[ \bigg\| \int_0^t S(t - t')\N(u)(t') dt'\bigg\|_{X^{s,b+2\dl}_{p;T}} 
\les T^{\delta} \| \N(u)\|_{X^{s,-\frac{1}{p}+2\dl}_{p;T}} 
\les \|u\|_{X^{s,b}_{p;T}}^{3} <\infty.\]

\noi
Since $b +2\dl = \frac{1}{p'} + \dl > \frac{1}{p'}$, 
we conclude from \eqref{conti} that the nonlinear part $u - S(t) u_0 +i \Psi$ is also continuous in time
with values in $\FL^{s, p}(\T)$.
Putting all together, we conclude that 
$u \in C([0,T];\FL^{s,p}(\T))$ almost surely.
This completes the proof of Theorem \ref{THM:1}.

\appendix

\section{On the $\g$-radonifying operators}\label{SEC:gamma}

In this appendix, we go over the basic definitions and properties
of $\g$-radonifying operators. 
These operators appear as  natural extensions of Hilbert-Schmidt operators
 to a Banach space setting. 
 A practical example is the theory of stochastic integrations 
in the Banach space setting,  where $\g$-radonifying operators suitably generalize
  the role played by Hilbert-Schmidt operators in the more ordinary Hilbert space setting 
  \cite{Br2, VN3}. 
  The content of this section is mostly taken from  the textbook \cite[Chapter 9]{HVNVW2}.

Let $H$ be a separable Hilbert space and 
$B$ be a Banach space.
We  denote the space of bounded linear operators from $H$ into $B$ by $\L(H;B)$. 
An operator $T\in \L(H;B)$ is said to be of finite rank if it can be represented as 
\[ T(\cdot)=\sum_{n=1}^{N}\jb{\,\cdot, e_{n}}x_{n},\]

\noi
where $\jb{\cdot, \cdot}$ denotes the inner product in $H$, $\{e_{n}\}_{n=1}^{N}$ is orthonormal in $H$, and 
$\{ x_{n}\}_{n=1}^{N}\subset B$. 

\begin{definition} \rm
We define the space $\g(H;B)$ as the closure of the set of finite rank operators in $\L(H;B)$ under the norm: 
\begin{equation}\label{g1}
\|T\|_{\g(H;B)}:=\sup_{N \in \NB} \Bigg( \E \bigg[ \bigg\| \sum_{n=1}^{N} g_{n}
T(e_n) \bigg\|_{B}^{2} \bigg] \Bigg)^{\frac{1}{2}}, 
\end{equation}

\noi
 where $\{e_{n}\}_{n\in \NB}$ is an orthonormal basis for $H$ and 
 $\{g_{n}\}_{n\in \NB}$ is  a sequence of independent standard complex-valued Gaussian random variables. 
 An operator  $T\in \g(H;B)$ is called a $\g$-radonifying operator.
\end{definition}
 
The definition \eqref{g1} is  independent of a choice of orthonormal basis $\{e_{n}\}_{n\in \NB}$ for $H$. 
In the following, it is understood that the results are independent of this choice. 
Note that  $(\g(H;B),\|\cdot\|_{\g(H;B)})$ is a Banach space and is separable whenever $B$ is separable.

\begin{remark}\rm 
The $L^{2}(\O)$-norm appearing in \eqref{g1} can be replaced with $L^{p}(\O)$ for any $1\leq p<\infty$. 
Strictly speaking,  this creates a family of spaces $\g_{p}(H;B)$, however,  
by the Kahane-Khintchine inequality (namely, the Banach-valued extension
of Lemma \ref{LEM:gauss}), all these norms are equivalent. 
Hence, it suffices to consider the most natural choice  $p=2$ and set $\g(H;B) : = \g_{2}(H;B)$.
\end{remark}

The property of being $\g$-radonifying is stable under transformations 
by bounded linear operators in either direction.

\begin{lemma}[Ideal property; Theorem 9.1.10 in \cite{HVNVW2}]\label{LEM:ideal} Let $H$ and $H'$ be Hilbert spaces and $B$ and $B'$ be Banach spaces. Let $U\in \L(H';H)$, $S\in \L(B;B')$,  and 
$T\in \g(H;B)$. 
Then,  $STU\in \g(H';B')$ and 
\[\|STU\|_{\g(H';B')}\leq \|S\| \|T\|_{\g(H;B)} \|U\|.\]
\end{lemma}

A very useful characterization of $\g$-radonifying operators is the following. 

\begin{proposition}[Theorem 9.1.17 in \cite{HVNVW2}] 
An operator $T\in \L(H;B)$ is $\g$-radonifying if and only if
  the sum $\sum_{n\in \NB}g_{n}T(e_{n})$ converges in $L^{p}(\O; B)$
  for some $1\leq p<\infty$  and  some orthonormal basis $\{ e_{n}\}_{n\in \NB}$ for $H$.
If $T\in \g(H;B)$, 
then this sum also converges almost surely and defines a $B$-valued Gaussian random variable with covariance operator $TT^{*}$. Furthermore,  we have 
\begin{equation}
\|T\|_{\g(H;B)}=\Bigg(\E \Bigg[\bigg\| \sum_{n\in \NB}g_{n}T(e_{n}) \bigg\|^{2}_{B} \Bigg] \Bigg)^{\frac{1}{2}}
\label{g2}
\end{equation}

\noi
for any orthonormal basis $\{ e_{n}\}_{n\in \NB}$ for $H$.
\end{proposition}

Recall that $T\in \L(H;H')$ is Hilbert-Schmidt from $H$ to Hilbert $H'$ if 
\begin{align*}
\|T\|_{HS(H;H')}^{2}:=\sum_{n\in \NB}\|T(e_{n})\|_{H'}^{2}<\infty
\end{align*}

\noi
for some (and hence any) orthonormal basis $\{ e_{n}\}_{n\in \NB}$ for $H$. 
It is then clear from \eqref{g2} that  when $B$ is a Hilbert space, 
then we have  $\g(H;B)=HS(H;B)$ and $\|T\|_{\g(H;B)}=\|T\|_{HS(H;B)}.$

Note that when $B$ is a Hilbert space, 
we can express the $\g(H;B)$-norm  {\it without} the use of probability. 
When $B$ is a Lebesgue space, 
we can also characterize the $\g(H;B)$-norm without the use of probability.

\begin{proposition}[Proposition 9.3.1 and (9.21) on p.\,285 in \cite{HVNVW2}]\label{PROP:g3} Let $(X,\mathcal{M},\mu)$ be a $\s$-finite measure space, $1\leq p<\infty$,  and  $T\in \L(H; L^p(X))$. 
Then,  $T\in \g(H;L^{p}(X))$ if and only if
the function $\big(\sum_{n\in \NB}|T(e_{n})|^{2} \big)^{\frac{1}{2}}$ lies in $L^{p}(X)$
 for some orthonormal basis $\{ e_{n}\}_{n\in \NB}$ for $H$.
 In this case, we have 
\begin{equation*}
\|T\|_{\g(H;L^{p}(X))}\sim  \bigg\|\bigg(\sum_{n\in \NB}|T(e_{n})|^{2} \bigg)^{\frac{1}{2}} \bigg\|_{L^{p}(X)}.
\end{equation*}
\end{proposition}

In Section \ref{SEC:intro}, 
we defined 
the  $\g(L^{2}(\T); \FL^{s,p}(\T))$-norm by \eqref{G1} 
for $s\in \R$ and $1 \leq p < \infty$.
Let us see how this definition \eqref{G1} appears
from the theory discussed above.
Let $\phi \in \L(L^{2}(\T); \FL^{s,p}(\T))$. 
Then,  from the definition \eqref{FL1} of the Fourier-Lebesgue space $\FL^{s, p}(\T)$, 
we have $\jb{\nabla}^{s}\phi \in \L(L^{2}(\T); \FL^{0,p}(\T))$ 
and $\F\jb{\nabla}^{s}\phi \in \L(L^{2}(\T); \l^{p}(\Z))$. 
From the ideal property (Lemma \ref{LEM:ideal}) and the invertibility of $\jb{\nabla}^{s}$ and $\F$, 
we see that $\phi\in \g(L^{2}(\T); \FL^{s,p}(\T))$ if and only if 
$\F\jb{\nabla}^{s}\phi\in \g(L^{2}(\T); \l^{p}(\Z))$. 
Furthermore, 
from Proposition~\ref{PROP:g3}, we have
\begin{align*}
\|\phi\|_{ \g(L^{2}(\T); \FL^{s,p}(\T))}
=\| \F\jb{\nabla}^{s}\phi\|_{\g(L^{2}(\T); \l^{p}(\Z))}
\sim \bigg\|\bigg(\sum_{k\in \Z}|\mathcal{F}(\jb{\nabla}^{s}\phi(e_{k}))(n)|^{2} \bigg)^{\frac{1}{2}} \bigg\|_{\l^{p}_{n}(\Z)},
\end{align*}

\noi
where $e_k (x) = e^{2\pi i k x}$, $k \in \Z$.
This justifies the definition  \eqref{G1}.

\begin{ack}\rm 

J.\,F.~was supported by the Maxwell Institute Graduate School in Analysis and its Applications, a Centre for Doctoral Training funded by the UK Engineering and Physical Sciences Research Council (grant EP/L016508/01), the Scottish Funding Council, Heriot-Watt University and the University of Edinburgh.
T.\,O.~was supported by the European Research Council (grant no.~637995 ``ProbDynDispEq'').

\end{ack}

\end{document}